\documentclass[12pt,leqno]{article}

\usepackage{graphicx}
\usepackage{amsmath}
\usepackage{amssymb}
\usepackage{amsthm}
\usepackage{mathrsfs}
\usepackage{verbatim} 

\usepackage[T1]{fontenc}
\usepackage{bbm}

\usepackage[english]{babel}
\usepackage{paralist}
\usepackage{fullpage}
\usepackage{color}
\usepackage{stmaryrd}

\usepackage{amsmath}
\usepackage{amsthm}
\usepackage{amssymb}
\usepackage{amsfonts}

\newtheorem{thm}{Theorem}%[section]
\newtheorem{lem}[thm]{Lemma}

\begin{document}

\begin{center}
\hspace*{-8mm}\mbox{\Large Numerical solution to the  Neumann problem in a Lipschitz domain,} \\
{\Large based on random walks}\\[3mm]

 %Lucian Beznea\footnote{%
 %Simion Stoilow Institute of Mathematics of the Romanian Academy, 
%Research unit No. 2,  
 %P.O. Box \mbox{1-764,} RO-014700 Bucharest, Romania, and 
 % POLITEHNICA Bucharest, CAMPUS Institute. {E-mail}: {\
 %lucian.beznea@imar.ro}}, 
{Oana Lupa\c scu-Stamate}\footnote{Institute of Mathematical Statistics and Applied Mathematics of the Romanian Academy,  
Calea 13 Septembrie 13, Bucharest, Romania.
%and Simion Stoilow Institute of Mathematics  of the Romanian Academy, Bucharest, Romania.
 E-mail: oana.lupascu$\_$stamate@ismma.ro, \\ $https://orcid.org/0009-0003-1787-5129$} and 
 {Vasile St\u anciulescu}\footnote{ %Institute of Mathematical Statistics and Applied Mathematics of the Romanian Academy,  
%Calea 13 Septembrie 13, Bucharest, Romania.
 E-mail: vasile.nicolae.stanciulescu@gmail.com, Bucharest, Romania.}
\end{center}

\vspace{2mm}

\begin{abstract}
We deal  with probabilistic numerical solutions for linear elliptic equations with
Neumann boundary conditions in a Lipschitz domain, by using a probabilistic numerical scheme introduced by
Milstein and Tretyakov based on new numerical layer methods.
%in \cite{MT} 
%to solve the linear elliptic PDEs with Dirichlet boundary conditions.
\end{abstract}

\vspace{1mm}

\noindent
{\bf Mathematics Subject Classification (2010):} 
35J25,     % Boundary value problems for second-order elliptic equations
65C05,    %  Monte Carlo methods
65C30,     %  Numerical solutions to stochastic differential and integral equations
65Nxx,     %  Numerical methods for partial differential equations, boundary value problems
65N75,     %  Probabilistic methods, particle methods, etc. for boundary value problems involving PDEs
60J50,    %  Boundary theory for Markov processes
60H30,      % Applications of stochastic analysis (to PDEs, etc.)
%60H15,     % Stochastic partial differential equations [See also 35R60]
%60H10,     % Stochastic ordinary differential equations [See also 34F05]
%35J65,       % Nonlinear boundary value problems for linear elliptic equations
%60J45,  	% Probabilistic potential theory
%60J35,    	% Transition functions, generators and resolvents
%60J40,  	% Right processes
%60J57,  	% Multiplicative functionals
%31C25,      % Dirichlet spaces
%47D07,  	% Markov semigroups and applications to diffusion processes 
%60J25,      % Continuous-time Markov processes on general state spaces
%60J60,      % Diffusion processes
%37C40 (primary), 37A30, 37L40, 60J35, 60J25, 60J60, 31C25, 37C40,  82B10 (secondary)
%47D03  	%Groups and semigroups of linear operators 
%47A35  	%Ergodic theory
%37A30   %Ergodic theorems, spectral theory, Markov operators
%37C40   %Smooth ergodic theory, invariant measures
%37L40   %Invariant measures (Infinite-dimensional dissipative dynamical systems)
%60J55  	Local time and additive functionals
%60J25  	Continuous-time Markov processes on general state spaces
%60J35: Transition functions, generators and resolvents
%82B10  	Quantum equilibrium statistical mechanics (general)
%31C05, 	%Harmonic, subharmonic, superharmonic functions
% 60J45,     %Probabilistic potential theory
% 60J60: Diffusion processes
60J65.  % Brownian motion

%35R60,       % Partial differential equations with randomness, stochastic partial differential equations
\section{Introduction}

\bigskip

Getting the solution of boundary value problems for partial differential
equations (PDEs) is a general topic in applied mathematics. Analytical
methods (e.g., separation of variables and transform techniques) are valued
for their exactness and the insight they provide; however, the range of
problems they solve is limited. Numerical schemes (e.g., finite element,
finite difference, and spectral methods) solve a much wider range of
problems. Some methods (e.g., boundary integral methods) combine specific
analytical information about the solution with numerical approximations. 

 In this paper  we solve  numerically a class of linear elliptic equations with
Neumann boundary conditions in a Lipschitz domain, by using a probabilistic numerical scheme introduced by
Milstein and Tretyakov (cf.  \cite{MT02} and \cite{MT}) to solve the linear elliptic PDEs, together with some results obtained  in  \cite{BH}. %and \cite{BP}; see also   \cite{BePaPaControl} for the case of discontinuous boundary data. 
 It is to be noted that as in \cite{MT} these methods are
probabilistic ones. 
Probabilistic representations of the solutions with new numerical layer methods for semilinear parabolic equations with Neumann conditions are constructed in \cite{MT02} and 
the convergence of the algorithm is proved. 
This technique was used in \cite{LuSt17} in order to give a probabilistic numerical approach for the nonlinear Dirichlet problem associated with a branching process.

Monte Carlo methods to compute the solution of elliptic equations with  Neumann boundary conditions are introduced in \cite{MaTa13}. 
The algorithm is based on the Euler scheme coupled with a local time approximation method and numerical examples are given on the Laplace operator on a square domain. Moreover, the authors have introduced a walk-on spheres approximation in the inhomogeneous Neumann boundary conditions.
The article \cite{RaBr20} points out some advances in probabilistic approximation methods to Neumann problem, including some numerical examples. 

A gradient Newton Galerkin numerical method to obtain a positive solution to an elliptic equation with Neumann boundary condition is introduced in \cite{AfMaNa07}
while discrete boundary element methods with error estimates for Neumann problems for the heat equation are presented in \cite{Ya99}. The authors in \cite{AgAe21} develop a numerical method to solve Neumann problem with discontinuous coefficients,  which  is illustrated through numerical examples with given exact solutions. 
In \cite{LiShGo07}  it is developed a wavelet method with explicit $C^1$ wavelet bases on Powell-Sabin triangulations to approximate the solution of the Neumann boundary problem for partial differential equations. In \cite{JuMa06} a modified implicit prediction domain decomposition algorithm is used to solve parabolic partial differential equations with Neumann boundary conditions. 

The structure of the paper is the following. 
Section 2  presents several preliminary results from \cite{BH}, on the Neumann boundary value problem in a bounded Lipschitz domain.  
In particular, Theorem \ref{thm1}  gives the probabilistic representation of the solution to the Neumann problem
by Brosamler’s formula.
This  provides the existence of a unique generalized solution to the Neumann problem. 
Next, we construct a Markov chain approximation for the solution using a boundary layer method inspired by \cite{MT} and we write down the corresponding algorithm. We give a convergence result in the dimensional space $d\geq 5$ for a function $f\in C^4(\bar{O})$, see Theorem \ref{thm4}. Finally, we presented same final remarks on our results.

\section{The Neumann problem in a Lipschitz domain}\label{sect2}

{

\noindent
{\bf Setting of the problem and existence result.} 
Consider  the following Neumann boundary value problem
on a bounded Lipschitz domain $O$

\begin{equation}\label{Ddpb}
  \left\{
 \begin{aligned}
 \Delta w =0 \quad\text{in}\text{}\text{
}\text{ }O, \; \; \\[1mm] 
{\hspace*{2mm}\frac{ \partial w}{ \partial \nu } =f \text{}\text{}\ \text{ on } \ \partial O,} \\
      \end{aligned}
    \right.
\end{equation}
where
$\frac{\partial\ }{\partial \nu}$ is the outward normal derivative to the boundary $\partial O$ of $O$,
$f$ is a bounded measurable functions on $\partial O$, 
%\in b\mathcal{B(\partial O)}$ 
%(the space of bounded measurable functions on $\partial O$) 
and we assume that
$\int_{\partial O} f d\sigma=0$; we denoted by $\sigma$ the surface measure on the boundary $\partial O$.
%${\textstyle\int _{ \partial O}}f(x)\sigma (dx) =0.$ 
Recall that this is a necessary
condition for the existence of a solution to the problem (\ref{Ddpb}), according to the Green’s first identity.}

\vspace{2mm}

A {\it generalized solution to the Neumann boundary value problem $(\ref{Ddpb})$}   is a 
function $w\in C(\overline{O})$ 
such that  
$$
\int_O w(x)\Delta \varphi (x)dx+\int_{\partial O} f(x)\varphi (x)\sigma(dx)=\int_{\partial O} w(x)\dfrac{ \partial \varphi}{ \partial \nu }(x) \sigma(dx)\ 
\mbox{ for any } \varphi\in C^2(\overline{O})
$$
(cf. e.g.  Definition 5.1 from \cite{BH}).
%\end{definition}

We recall now the existence result of the generalized solution to the Neumann boundary value problem, its probabilistic representation using the reflecting Brownian motion $X$ on $O$, and some estimates for  the transition density $p(t,x,y)$  of the reflecting Brownian motion; cf. \cite{BH}, Theorem 5.3 and Theorem 3.1.

\begin{thm}\label{thm1}
 Let $O$ be a bounded Lipschitz domain. 
 Then the following assertions hold.
 
$(i)$  If $f\in\mathcal{B}(\partial O)$ with 
 $\int_{\partial O} f d\sigma=0$
 then there exists a unique generalized solution to the Neumann boundary problem \eqref{Ddpb} satisfying the condition 
${\int _{O}}w(x)dx =0$, where $\mathcal{B}(\partial O)$ is a Borel $\sigma$-algebra on $\partial O$. 
In addition we have   %for each $x \in O$,
\begin{equation}\label{eq2}
w(x) =\lim _{t \rightarrow \infty }\frac{1}{2}\mathbb{E}^{x}
{\textstyle\int _{0}^{t}}f(X_{s})dL_{s}\ 
\mbox{ for each } x\in O, 
\end{equation}
where $L_t$ is the boundary local time for the reflecting Brownian motion $X$.

$(ii)$ 
There exist two positive constants $C_1,C_2>0$ such that for all $t>0$ and $x,y\in\overline{O}$ we have
$$
p(t,x,y)\leq C_1\cdot t^{-\dfrac{d}{2}}\cdot e^{\dfrac{-|x-y|^2}{C_2 t}}.
$$
\end{thm}

Relation (\ref{eq2}) is the {\it Brosamler formula}, the probabilistic representation of the solution of the Neumann problem, in terms of the reflecting Brownian motion;
see \cite{BP-Brosamler} for an approach to the Brosamler formula on balls, using the connection
between the Dirichlet and the Neumann boundary problems from \cite{BP}, and the explicit description of the reflecting Brownian motion and its boundary local time in terms of
the free Brownian motion. See also   \cite{BePaPaControl} for the case of discontinuous boundary data. \\

{}

\vspace{5mm}
\noindent
{\bf Numerical approximation of the solution in a Lipschitz domain.}
Now let 
\begin{equation}\label{eqq1}
Y_{t ,y ,h} \approx Y =y +h^{\frac{1}{2}}\xi  ,
\end{equation}
where $h >0$ is a step of integration and $\xi  =(\xi ^{1} , . . . ,\xi ^{n}$$)^{T}$, $\xi ^{i}$, $i =1 , . . . ,n$ are mutually independent random variables taking values $ \pm \frac{1}{2}$ with probability $\frac{1}{2}$.

We define a boundary zone $S_{h} \subset \overline{O} :y \in S_{h}$ if at least one of the $2^{n}$ values of the vector $Y$ is outside $\overline{O\,}$.

Let $\lambda>0$ be a constant such that if the distance from $y \in O$ to the boundary $\partial O$ is equal to or greater than $\lambda \sqrt{h}$ then $y$ is outside the boundary zone and therefore, for such $y$ all the realizations of the random variable $Y$ belong to $\overline{O\,}$.

Let $y \in S_{h}$ and  we construct the random vector $Y_{y ,h}^{\pi }$ taking two values $y^{\pi }$ and $y^{\pi } +\lambda \eta (y^{\pi })h^{\frac{1}{2}}$ with probabilities $p =p_{Y ,h\text{}}$ and $q =q_{Y ,h} =1 -p_{Y ,h}$ respectively, where
\begin{equation*}p_{Y ,h} : =\frac{h^{\frac{1}{2}}\lambda }{\vert y +h^{\frac{1}{2}}\lambda \eta (y^{\pi }) -y^{\pi }\vert } ,
\end{equation*}
$y^{\pi } \in  \partial O$ is the projection of the point y on the boundary $ \partial O$ and $\eta (y^{\pi })$ is the unit vector of the internal normal to $ \partial O$ at the point $y^{\pi } .$

Let $h : =\frac{1}{N} ,$ $N >0 ,$ integer.

Further, we approximate the solution to the Neumann boundary value problem \eqref{Ddpb}, by constructing a
Markov chain $Y_k$ which stops when it reaches the boundary $\partial O$ at a random step $\varkappa$.

We set $Y_{0}^{ \prime } =x .$ If $Y_{0}^{ \prime } \notin S_{h}$ we take 
$$Y_{0} =Y_{0}^{ \prime } .$$

If $Y_{0}^{ \prime } \in S_{h}$ then the random variable $Y_{0}$ takes two values: either $Y_{0}^{ \prime \pi } \in  \partial O$ with probability $p_{Y_{0}^{ \prime \pi }}$ or $Y_{0}^{ \prime \pi } +h^{\frac{1}{2}}\lambda \eta (Y_{0}^{ \prime \pi }) \notin S_{h}$ with probability $q_{Y_{0}^{ \prime \pi }} .$ If $Y_{0} =Y_{0}^{ \prime \pi }$ we put $\varkappa  =0 ,$ $Y_{\varkappa } =Y_{0}^{ \prime \pi }$ and the random walk is finished.

Let $Y_{k} ,$ $k <N$ from above and either $Y_{k} \in $$ \partial O$ or $Y_{k} \notin S_{h} .$ We assume the chain does not stopped until step k i.e. $\varkappa  >k .$ 

We introduce $Y_{k +1}^{ \prime }$ due to \eqref{eqq1} with $t=t_k, y=Y_k, \xi=\xi_k$, as:

\begin{equation}\label{eqq2}
Y_{k +1}^{ \prime } =Y_{k} +h^{\frac{1}{2}}\xi _{k} .
\end{equation}

Now we obtain $Y_{k+1}$ using $Y^{\prime_{k+1}}$ as we got $Y_0$ using $Y^{\prime}_0$. More precisely, we use the following rule.

If $Y_{k +1}^{ \prime } \notin S_{h}$ then we take 
$$Y_{k +1} =Y_{k +1}^{ \prime } .$$
If $Y_{k +1}^{ \prime } \in S_{h}$ then the random variable $Y_{k +1}$ takes two values: either $Y_{k +1}^{ \prime \pi } \in  \partial O$ with probability $p_{Y_{k +1}^{ \prime \pi } ,h}$ or $Y_{k +1}^{ \prime } +h^{\frac{1}{2}}\lambda \eta (Y_{k +1}^{ \prime \pi }) \notin S_{h}$ with probability $q_{Y_{k +1}^{ \prime \pi } ,h} .$

If $Y_{k +1} =Y_{k +1}^{ \prime \pi }$ we put $\varkappa  =k +1 ,$ $Y_{\varkappa } =Y_{k +1}^{ \prime \pi } $ and the random walk is finished. 

So, the random walk $Y_k$ has been constructed and clearly, $Y_k$ remains in the domain $\bar{O}$ with probability $1$.

Finally, we introduce an extended  Markov chain defined  as $Y_{k} =Y_{\varkappa }$ for $k >\varkappa  .$\\

The constructed algorithm from above can be written as follows.

\vspace{2mm}
\vspace{2mm}
\noindent {\bf Algorithm 1}\\[3mm]
\vspace{2mm}{\tt\noindent{\bf STEP 0}: $Y_{0}^{\prime }=x_{0}$\\
\noindent{\bf STEP 1}: If $Y_{k}^{\prime }\notin S_{h}$ then $Y_{k} = Y_{k}^{\prime }$ and go
to {\bf STEP 3}.

\ \ \ \ \ \ \ \ \ \ \ \ \ \ \ If $Y_{k}^{\prime }\in S_{h}$ then either 
$Y_{k}$=$Y_{k}^{\prime \pi }$ with probability $p_{Y_{k}^{\prime },h}$ or 
\;\;\;\;\;\;\;\; $Y_{k} = Y_{k}^{\prime }$ $+h^{1/2}\lambda \eta(Y_{k}^{\prime \pi })$ with
probability $q_{Y_{k}^{\prime },h}$.\\[2mm]
\noindent{\bf STEP 2}: If $Y_{k}$ = $Y_{k}^{\prime \pi }$ then {\bf STOP} and $\varkappa =k$, $%
X_{\varkappa }$=$Y_{k}^{\prime \pi }$;\\[2mm]
\noindent{\bf STEP 3}: Sampling $\xi _{k}$ and compute $Y_{k+1}^{\prime }$ by using \eqref{eqq2}, for $y=Y_{k}$, $\xi
=\xi _{k}.$\\[2mm]
\noindent{\bf STEP 4}: Set $k := k + 1$ and go to {\bf STEP 1}.}

\vspace{10mm}

\noindent
{\bf Convergence results.} The following results  ensures the convergence of the above algorithm.\\

We assume that $f\in C^4(\overline{O})$.

\begin{lem}\label{lem1}
We have

\begin{equation}
\left|\dfrac{1}{2}\lim_{s\to\infty}\left[\mathbb{E}^{x}{\textstyle\displaystyle\int _{0}^{s }}f(Y_{\varkappa }) dL_{t} - \mathbb{E}^{x}{\textstyle\displaystyle\int _{0}^{s }}f(X_t)dL_{t}\right]\right|\leq Ch,
\end{equation}
where the constant $C$ does not depend on $x$ and $h$.
\end{lem}

\begin{proof}
{ The result follows reasoning as in the proof of Theorem 7.3.4 from \cite{MT}.}
\end{proof}

\vspace{3mm}

The convergence result is given by the  following theorem.

\begin{thm} \label{thm4}
Let 
\begin{equation}\label{ap1}
Z_{\varkappa}(x) : =\frac{1}{2}\lim_{s\to\infty}\mathbb{E}^{x}\int_0^s f(Y_{\varkappa })\mbox{d}L_{t}.
\end{equation}
Assume that $d\geq 5.$  
Then there exists a constant $C>0$ which does not depend on $x$ and $h$, such that 
\begin{equation*}
\vert Z_{\varkappa }(x) -w(x)\vert  \leq Ch.
\end{equation*}
%where the constant C does not depend on $x$ and $h$.
\end{thm}

\begin{proof}
By assertion $(i)$ of  Theorem \ref{thm1} we have 
%$|a+b|\leq |a|+|b|$

\begin{equation*}\label{eq1th4}
\vert Z_{\varkappa }(x) -w(x)| \leq \frac{1}{2}\left|\lim_{s\to\infty}\mathbb{E}^{x}\int_0^s [f(Y_{\varkappa })-f(x)]\mbox{d}L_{t}\right| +\frac{1}{2}\left|\lim_{s\to\infty}\mathbb{E}^{x}\int_0^s [f(X_t )-f(x)]\mbox{d}L_{t}\right|. 
\end{equation*}
Consequently, from Lemma \ref{lem1} we obtain now
\begin{equation}\label{eq2th4}
\vert Z_{\varkappa }(x) -w(x)| \leq Ch+\frac{1}{2}\left|\lim_{s\to\infty} \int_0^s dt\int_{\partial O}p(t,x,y)\cdot[f(y)-f(x)]\sigma(dy)\right|.
\end{equation}
Further, using Lagrange's inequality we get 

\begin{equation*}
\vert Z_{\varkappa }(x) -w(x)| \leq Ch+\frac{1}{2}\sup_{y\in\bar{O}}|\nabla f(y)|\cdot\lim_{s\to\infty}\int_0^s \alpha t\int_{\partial O}p(t,x,y)\cdot |y-x|\sigma(dy).
\end{equation*}
By assertion $(ii)$ of Theorem \ref{thm1}, we obtain 
\begin{equation*}
\int_{\partial O}p(t,x,y)\cdot |y-x|\sigma(dy)\leq C_1\cdot  t^{-\dfrac{d}{2}}\int_{\partial O}e^{\dfrac{-|y-x|^2}{C_2t}}|y-x|d\sigma(y)\leq \alpha\cdot t^{-\dfrac{d}{2}} \cdot e^{-\dfrac{\beta}{t}},
\end{equation*}
where
$\alpha$ and $\beta$ are positive constants.

From \eqref{eq2th4}  there exist $\alpha, \beta >0$ such that 
\begin{equation}\label{eq3th4}
\vert Z_{\varkappa }(x) -w(x)| \leq Ch+ K_1 \lim_{s\to\infty}\int_0^s t^{-\dfrac{d}{2}} \alpha e^{-\dfrac{\beta}{t}} dt.
\end{equation}
Because we supposed that $d\geq 5 $, it follows that 
\begin{equation}\label{eq4th4}
0 \leq  \lim_{s\to \infty}\int_0^s t^{-\dfrac{d}{2}} \cdot e^{-\dfrac{\beta}{t}} dt\leq K_2 h,
\end{equation}
where $K_2$ is a positive constant which does not depend on $x$ and $h$.

From \eqref{eq3th4} and \eqref{eq4th4} we deduce  the claimed inequality,
\begin{equation*}
\vert Z_{\varkappa }(x) -w(x)|\leq (C+K_1+K_2)h.
\end{equation*}
\end{proof}

\subsection*{Final remarks}
We complete the paper with several concluding remarks.
Recall that our aim was to present a numerical treatment of some classes of Neumann problems for linear elliptic PDEs in a Lipschitz domain. The problem was considered in the paper as follows.

Revision of existence in the generalized sense for the solution to the Neumann problem for linear elliptic PDEs  together with a result from Bass and Hsu (see \cite{BH}). 
Also, the construction of new approximations of solutions to the Neumann problem for linear elliptic PDEs was obtained by combining the result from Bass and Hsu from \cite{BH} with the random walk technique from \cite{MT}.

In a forthcoming work, we intend to investigate  numerical methods for the linear elliptic Neumann problem for PDEs in a unit ball,  by exploiting the mentioned equivalence result together with the ideas of the simplest random walks from \cite{MT}. 
We  shall use an equivalence result between the solution to the linear elliptic Neumann problem for PDEs and the solution to the linear elliptic Dirichlet problem; cf. \cite{BP}. It is a challenge to give approximations for the solutions for some classes of parabolic SPDEs with nonlinear Neumann boundary conditions, considered for example in \cite{BB}.

%Also, in the second part of the paper we have revised some results concerning the existence and uniqueness of the classical solution to the Neumann problem for linear elliptic PDEs and we presented the probabilistic representations of its solution. 

\vspace{3mm}

\noindent \textbf{Acknowledgements.} 
This work was supported by a grant of the Ministry of Research, Innovation and Digitization, CNCS - UEFISCDI,
project number PN-III-P4-PCE-2021-0921, within PNCDI III.

\end{document}